\documentclass{amsart}
\usepackage{amssymb}
\usepackage{amscd}
\usepackage{amssymb}
\usepackage{amsthm}

\input xy
\xyoption{all}

\newcommand{\Cl}{\mathrm{Cl}}

\newcommand{\id}{\mathrm{id}}
\newcommand{\pr}{\mathrm{pr}}
\newcommand{\M}{\mathcal M}

\newcommand{\A}{\mathcal A}
\newcommand{\B}{\mathcal B}

\newcommand{\N}{\mathbb N}
\newcommand{\R}{\mathbb R}

\newtheorem{theorem}{Theorem}[section]
\newtheorem{corollary}[theorem]{Corollary}

\newtheorem{lemma}[theorem]{Lemma}
\newtheorem{question}{Question}
\newtheorem*{theoremA}{Theorem A}

\theoremstyle{definition}
\newtheorem{df}[theorem]{Definition}

\title{Idempotent measures:absolute retracts and soft maps}

\author[T.~Radul]{Taras Radul}
\address[T.~Radul]{Kazimierz Wielki University, Bydgoszcz (Poland) and Ivan Franko National University of Lviv (Ukraine)}
\email{tarasradul@yahoo.co.uk}


\begin{document}
\begin{abstract} We investigate under which conditions the space of idempotent measures is an absolute retract and  the idempotent barycenter map is soft.
\end{abstract}

\maketitle

\section{Introduction}

The notion of idempotent (Maslov) measure finds important applications in different
part of mathematics, mathematical physics and economics (see the survey article
\cite{Litv} and the bibliography therein). Topological and categorical properties of the functor of idempotent measures were studied in \cite{Zar}. Although idempotent measures are not additive and corresponding functionals are not linear, there are some parallels between topological properties of the functor of probability measures and the functor of idempotent measures (see for example \cite{Zar} and \cite{Radul}) which are based on existence of  natural equiconnectedness structure on both functors.

However, some differences appear when the problem of the openness of the barycentre map was studying.
The problem of the openness of the barycentre map of probability measures was investigated  in \cite{Fed}, \cite{Fed1},  \cite{Eif}, \cite{OBr} and \cite{Pap}. In particular, it is proved in \cite{OBr} that the barycentre map for a compact convex set  in a locally convex space is open iff the map $(x, y)\mapsto 1/2 (x + y)$ is open.
Zarichnyj defined in \cite{Zar} the idempotent barycentre map for idempotent measures and asked if the analogous characterization is true. A negative answer to this question was given  in  \cite{Radul1}.

We investigate the problem when the space of idempotent measures is absolute retract (shortly AR). It is shown in \cite{Zar} that the space of idempotent measures $I([0,1]^\tau)$ on Tychonov cube $[0,1]^\tau$ is not an absolute retract for any $\tau>\omega_1$. It follows from the results of \cite{Radul} that the space of idempotent measures $IX$  is  an absolute retract for each openly generated compactum $X$  of the weight $\le\omega_1$. We will show in this paper that the space of idempotent measures $IX$  is  an absolute retract iff $X$ is an openly generated compactum   of the weight $\le\omega_1$. Let us remark that it is an idempotent analogue of Ditor-Haydon Theorem for probability measures \cite{DH}.

The problem of the softness  of the barycentre map of probability measures was investigated  in \cite{Fed1}, \cite{Radul2} and \cite{Radul3}. Fedorchuk proved in \cite{Fed1} that each product of $\omega_1$ barycentrically open convex metrizable compacta (i.e.  convex metrizable compacta for which the barycentre map  is open) is barycentrically soft and asked two questions: if each barycentrically open convex compactum  of the weight $\le\omega_1$ is baricentrically soft and if there exists a baricentrically soft convex compactum  of the weight $\ge\omega_2$. The first question was answered in negative in \cite{Radul2}, showing that barycenrical softness of the space of probability measures $PX$ implies metrizability of the compactum $X$. The second question was answered in negative in \cite{Radul3}.

In this paper we discuss analogous problems for the space of idempotent measures and idempotent barycenter map.

\section{Idempotent measures: preliminaries}

In the sequel, all maps will be assumed to be continuous. Let $X$ be a compact Hausdorff space. We shall denote by $C(X)$ the
Banach space of continuous functions on $X$ endowed with the sup-norm. For any $c\in\ R$ we shall denote  by $c_X$ the
constant function on $X$ taking the value $c$.

Let $\R_{\max}=\R\cup\{-\infty\}$ be the metric space endowed with the metric $\varrho$ defined by $\varrho(x, y) = |e^x-e^y|$.
 Following the notation of idempotent mathematics (see e.g., \cite{MS}) we use the
notations $\oplus$ and $\odot$ in $\R$ as alternatives for $\max$ and $+$ respectively. The convention $-\infty\odot x=-\infty$ allows us to extend $\odot$ and $\oplus$  over $\R_{\max}$.

Max-Plus convex sets were introduced in \cite{Z}.
Let $\tau$ be a cardinal number. Given $x, y \in \R^\tau$ and $\lambda\in\R_{\max}$, we denote by $y\oplus x$ the coordinatewise
maximum of x and y and by $\lambda\odot x$ the vector obtained from $x$ by adding $\lambda$ to each of its coordinates. A subset $A$ in $\R^\tau$ is said to be  Max-Plus convex if $\alpha\odot a\oplus  b\in A$ for all $a, b\in A$ and $\alpha\in\R_{\max}$ with $\alpha\le 0$. It is easy to check that $A$  is   Max-Plus convex iff $\oplus_{i=1}^n\lambda_i\odot\delta_{x_i}\in A$ for all $x_1,\dots, x_n\in A$ and $\lambda_1,\dots,\lambda_n\in\R_{\max}$ such that $\oplus_{i=1}^n\lambda_i=0$. In the following by Max-Plus convex compactum we mean a Max-Plus convex compact subset of $\R^\tau$.

We denote by $\odot:\R\times C(X)\to C(X)$ the map acting by $(\lambda,\varphi)\mapsto \lambda_X+\varphi$, and by $\oplus:C(X)\times C(X)\to C(X)$ the map acting by $(\psi,\varphi)\mapsto \max\{\psi,\varphi\}$.

\begin{df}\cite{Zar} A functional $\mu: C(X) \to \R$ is called an idempotent  measure (a Maslov measure) if

\begin{enumerate}
\item $\mu(1_X)=1$;
\item $\mu(\lambda\odot\varphi)=\lambda\odot\mu(\varphi)$ for each $\lambda\in\R$ and $\varphi\in C(X)$;
\item $\mu(\psi\oplus\varphi)=\mu(\psi)\oplus\mu(\varphi)$ for each $\psi$, $\varphi\in C(X)$.
\end{enumerate}

\end{df}

Let $IX$ denote the set of all idempotent  measures on a compactum $X$. We consider
$IX$ as a subspace of $\R^{C(X)}$. It is shown in \cite{Zar} that $IX$ is a compact Max-Plus subset of $\R^{C(X)}$. The construction $I$ is  functorial what means that for each continuous map $f:X\to Y$ we can consider a continuous map $If:IX\to IY$ defined as follows $If(\mu)(\psi)=\mu(\psi\circ f)$ for $\mu\in IX$ and $\psi\in C(Y)$. It is proved in \cite{Zar} that the functor $I$ preserves topological embedding. For an embedding $i:A\to X$ we shall identify the space $F(A)$ and the subspace $F(i)(F(A))\subset F(X)$.

By $\delta_{x}$ we denote the Dirac measure supported by the point $x\in X$. We can consider a map $\delta X:X\to IX$ defined as $\delta X(x)=\delta_{x}$, $x\in X$. The map $\delta X$ is continuous, moreover it is an embedding \cite{Zar}. It is also shown in \cite{Zar} that the set $$I_\omega X=\{\oplus_{i=1}^n\lambda_i\odot\delta_{x_i}\mid\lambda_i\in\R_{\max},\ i\in\{1,\dots,n\},\ \oplus_{i=1}^n\lambda_i=0,\ x_i\in X,\ n\in\N\},$$ (i.e., the set of idempotent probability measures of finite support) is dense in $IX$.

Let $A\subset  \R^T$ be a compact max-plus convex subset. For each $t\in T$ we put $f_t=\pr_t|_A:A\to \R$ where $\pr_t:\R^T\to\R$ is the natural projection.    Given $\mu\in A$, the point $\beta_A(\mu)\in\R^T$ is defined by the conditions $\pr_t(\beta_A(\mu))=\mu(f_t)$ for each $t\in T$. It is shown in \cite{Zar} that $\beta_A(\mu)\in  A$ for each $\mu\in I(A)$ and the map $\beta_A : I(A)\to A$ is continuous.
The map $\beta_A$ is called the idempotent barycenter map.

For a function $\varphi\in C(X)$ by $\tilde\varphi\in C(IX)$ we denote the function defined by the formula $\tilde\varphi(\nu)=\nu(\varphi)$ for $\nu\in IX$. Diagonal product $(\tilde\varphi)_{\varphi\in C(X)}$ embeds $IX$ into $\R^{C(X)}$ as a Max-Plus convex subset. It is easy to  see that the map $\beta_{IX}$ satisfies the equality $\beta_{IX}(\M)(\varphi)=\M(\tilde\varphi)$ for any $\M\in I^2X=I(IX)$ and $\varphi\in C(X)$. Particularly we have $\beta_{IX}\circ I(\delta X)=\id_{IX}$ for each compactum $X$.

A map $f:X\to Y$ between Max-Plus convex compacta $X$ and $Y$ is called Max-Plus affine if for each  $a, b\in X$ and $\alpha\in[-\infty,0]$ we have $f(\alpha\odot a\oplus  b)=\alpha\odot f(a)\oplus  f(b)$. It is easy to check that the diagram
$$ \CD
IX  @>If>>   IY \\
@VV\beta_XV      @VV\beta_YV  \\
X  @>f>>   Y
\endCD
$$
is commutative provided $f$ is Max-Plus affine. It is also easy to check that the map $b_X$  is Max-Plus affine for each Max-Plus convex compactum $X$ and the map $If$ is Max-Plus affine for each continuous map $f:X\to Y$ between  compacta $X$ and $Y$.

The notion of density for an idempotent measure was introduced in \cite{A}. Let $\mu\in IX$. Then we can define a function $d_\mu:X\to [-\infty,0]$ by the formula $d_\mu(x)=\inf\{\mu(\varphi)|\varphi\in C(X)$ such that $\varphi\le 0$ and $\varphi(x)=0\}$, $x\in X$. The function $d_\mu$ is upper semicontinuous and is called the density of $\mu$. Conversely, each upper semicontinuous function $f:X\to [-\infty,0]$ with $\max f = 0$ determines an idempotent measure $\nu_f$
by the formula $\nu_f(\varphi) = \max\{f(x)\odot\varphi(x) | x \in X\}$, for $\varphi\in C(X)$.

Let $A$ be a closed subset of a compactum $X$.  It is easy to check that $\nu\in IA$ iff $\{x\in X|d_\nu(x)>-\infty\}\subset A$.

\begin{lemma}\label{supp} Let $A$ be a closed subset of a compactum $X$. Then  $\beta_{IX}^{-1}(IA)\subset I^2A$.
\end{lemma}

\begin{proof} Suppose the contrary. Then there exists  $\M\in \beta_{IX}^{-1}(IA)\setminus I^2A$. Hence there exists  $\mu\in IX$ such that $d_\M(\mu)=s>-\infty$ and $d_\mu(x)=t>-\infty$ for some point $x\in X\setminus A$.
Choose a function $\varphi\in C(X)$ such that $\varphi(x)=1-s-t$ and $\varphi(A)\subset\{0\}$. Then we have $\M(\tilde\varphi)\ge\tilde\varphi(\mu)+s=\mu(\varphi)+s\ge\varphi(x)+t+s=1$. On the other hand $\beta_{IX}(\M)\in IA$ implies $\beta_{IX}(\M)(\varphi)=0$. But $\M(\tilde\varphi)=\beta_{IX}(\M)(\varphi)$ and we obtain a contradiction.
\end{proof}

\begin{lemma}\label{mapdenc}  Let $f:X\to Y$ be a continuous map,  $\nu\in IX$. Then  $d_{If(\nu)}(y)=\max\{d_\nu(x)|x\in f^{-1}(y)\}$ for each $y\in Y$.
\end{lemma}

\begin{proof}  Let $d:Y\to [-\infty,0]$ be a function defined by the formula $d(y)=\max\{d_\nu(x)|x\in f^{-1}(y)\}$ for  $y\in Y$. It is easy to see that the function $d$ is upper semicontinuous with  $\max d = 0$. Let $\mu$ be an idempotent measure generated by $d$. Then we have
$$\mu(\varphi)=\max\{d(y)+\varphi(y)|y\in Y\}=\max\{\varphi(y)+\max\{ d_\nu(x)|x\in f^{-1}(y)\}|y\in Y\}=$$
$$=\max\{\varphi\circ f(x)+ d_\nu(x)|x\in X\}=\nu(\varphi\circ f)=If(\nu)(\varphi)$$
for each $\varphi\in C(Y)$. Hence $\mu=If(\nu)$.
\end{proof}

\begin{lemma}\label{supp+}  Let $f:X\to Y$ be a continuous map, $A$ and $B$ are disjoint closed subsets of $Y$ and $\mu\in IX$  such that   $If(\mu)=s\odot\nu\oplus\pi$ where $\nu\in IA$ and $\pi\in IB$. Then there exist $\nu'\in I(f^{-1}(A))$ and $\pi'\in I(f^{-1}(B))$ such that $\mu=s\odot\nu'\oplus\pi'$.
\end{lemma}

\begin{proof} Consider the density $d_\mu$ of $\mu$. We have that $\max\{d_\mu(x)|x\in f^{-1}(A)\}=s$ and $\max\{d_\mu(x)|x\in f^{-1}(B)\}=1$ by Lemma \ref{mapdenc}. Consider functions $d_1,d_2:X\to [-\infty,0]$ defined by the formulas
$$d_1(x)=\begin{cases}
d_\mu(x)-c,&x\in A,\\
-\infty,&x\notin A\end{cases}
$$
and
$$d_2(x)=\begin{cases}
d_\mu(x),&x\in B,\\
-\infty,&x\notin B\end{cases}
$$
and idempotent measures $\nu'$ and $\pi'$ generated by function $d_1$ and $d_2$. Then $\nu'$ and $\pi'$ are the measures we are looking for.
\end{proof}

\section{Idempotent measures and absolute retracts}

By $w(X)$ we denote the weight of the space $X$ and by $\chi(X)$ the character of the space $X$.

We will need some notations and facts from the theory of
non-metrizable compacta. See \cite{Shchep} for more details.

Let $\tau$ be an infinite cardinal number. A partially ordered set $\A$ is
called $\tau$-{\it complete}, if every subset of cardinality $\le\tau$ has
a least upper bound in $\A$. An inverse system consisting of compacta and
surjective bonding maps over a $\tau$-complete indexing set is called
$\tau$-complete. A continuous $\tau$-complete system consisting of compacta
of weight $\le\tau$ is called a $\tau$-{\it system}.

As usual, by $\omega$ we denote the countable cardinal number, by $\omega_1$ we denote the first uncountable cardinal number and so on.

A compactum $X$ is called openly generated if $X$ can be represented as the
limit of an $\omega$-system with open bonding maps. We have $w(X)=\chi(X)$  for each openly generated compactum $X$ (see for example Lemma 4 from \cite{Radul4}). A compactum $X$ is called absolute extensor in the class of 0-dimensional compacta (shortly AE(0)) if for any 0-dimensional compactum $Z$, any closed subspace $A$ of $Z$ and a continuous map $\varphi:A\to X$  there
exists a continuous map $\Phi:Z\to X$ such that $\Phi|A=\varphi$. Evidently each absolute retract is AE(0). Let us also remark that  each  AE(0) is openly generated and these classes coincide  for compacta  of the weight $\le\omega_1$.

By $D$ we denote the two-point set with discrete topology.

\begin{lemma}\label{Domega} The compactum $I(D^\tau)$ is not an absolute retract for each $\tau\ge\omega_2$.
\end{lemma}

\begin{proof} Suppose the contrary: there exists $\tau\ge\omega_2$ such that $I(D^\tau)$ is  an absolute retract.
Choose a continuous onto map $f:D^\tau\to [0,1]^\tau$ such that   there exists a continuous map $s:[0,1]^\tau\to I(D^\tau)$ such that $If\circ s=\delta [0,1]^\tau$. Existence of such map follows from Theorem 2.1 \cite{Radul}.

Then we have $If\circ\beta_{I(D^\tau)}\circ Is=\beta_{I([0,1]^\tau)}\circ I^2f\circ Is=\beta_{I([0,1]^\tau)}\circ I(\delta [0,1]^\tau)=\id_{[0,1]^\tau}$. Hence the map $If:I(D^\tau)\to I([0,1]^\tau)$ is a retraction and the compactum $I([0,1]^\tau)$ is an absolute retract. We obtain a contradiction to the above mentioned Zarichnyi result.
\end{proof}

\begin{theorem}\label{Gen} The compactum $IX$ is an absolute retract iff $X$ is an openly generated compactum of the weight $\le\omega_1$.
\end{theorem}

\begin{proof} The sufficiency follows from Corollary 3.5 \cite{Radul} and the fact that the functor of idempotent measures preserves weight of infinite compacta, open maps and preimages \cite{Zar}.

Let us prove the necessity. Consider any compactum $X$ such that the compactum $IX$ is an absolute retract. Since the functor $I$ is normal \cite{Zar}, the compactum $X$  is AE(0) (\cite{Shchep}, Corollary 4.2).   Let us show that $w(X)\le\omega_1$. Suppose the contrary $w(X)>\omega_1$, then by Theorem 5.6 and Proposition 6.3 from  \cite{Hayd}, there exists an embedding $s:D^{\omega_2}\to X$. It follows from results of \cite{BR} that there exists a continuous map $f:X\to I(D^{\omega_2})$ such that $Is(f(x))=\delta_x$ for each $x\in s(D^\tau)$. Since the map $Is$ is an embedding, we have $f\circ s=\delta D^{\omega_2}$.

Define a map $u:C(D^{\omega_2})\to C(X)$ by the formula $u(\varphi)(x)=f(x)(\varphi)$ for $\varphi\in C(D^{\omega_2})$ and $x\in X$. It is easy to check that $u$ is well-defined, continuous and preserves operations $\odot$, $\oplus$ and constant functions. The equality $f\circ s=\delta D^{\omega_2}$ implies $u(\varphi)\circ s=\varphi$.

Define a map $\phi:IX\to I(D^{\omega_2})$ by the formula $\phi(\nu)(\varphi)=\nu(u(\varphi))$ for $\varphi\in C(D^{\omega_2})$ and $\nu\in IX$. Since $u$  preserves operations $\odot$, $\oplus$ and constant functions, $\phi(\nu)\in I(D^{\omega_2})$ for each $\nu\in IX$. It is easy check that $\phi$ is continuous.

Finally, for each $\varphi\in C(D^{\omega_2})$ and $\nu\in I(D^{\omega_2})$ we have $(\phi\circ Is)(\nu)(\varphi)=Is(\nu)(u(\varphi))=\nu(u(\varphi)\circ s)=\nu(\varphi)$. Hence the map $\phi$ is a retraction and $ I(D^{\omega_2}$ is an absolute retract. We obtain a contradiction to Lemma \ref{Domega}.

\end{proof}

\section{On the  softness of the idempotent barycenter map}

A map $f:X\to Y$ is said to be (0-)soft if for any (0-dimensional) paracompact space $Z$, any closed subspace $A$ of $Z$ and
maps $\Phi:A\to X$ and $\Psi:Z\to Y$ with $\Psi|A=f\circ\Phi$ there exists a
map $G:Z\to X$ such that $G|A=\Phi$ and $\Psi=f\circ G$. This notion is
introduced by E.Shchepin \cite{Shchep1}. Let us remark that each 0-soft map is open and 0-softness is equivalent to the openness for all the maps between metrizable compacta.

Let $$ \CD
X_1  @>p>>   X_2 \\
@VVf_1V      @VVf_2V  \\
Y_1  @>q>>   Y_2
\endCD
$$
be a commutative diagram. The map $\chi:X_1\to
X_2\times{}_{Y_2}Y_1= \{(x,y)\in X_2\times Y_1\mid f_2(x)=q(y)\}$
defined by $\chi(x)=(p(x),f_1(x))$ is called a characteristic map
of this diagram. The diagram is called open (0-soft, soft)  if the map $\chi$ is open (0-soft, soft).

The following theorem from \cite{Zar3} gives a
characterization of 0-soft maps:

\begin{theoremA}\cite{Zar3} A map $f:X\to Y$ is
0-soft if and only if there exist $\omega$-systems $S_X$ and $S_Y$ with the
limits $X$ and $Y$ respectively and a morphism $\{f_\alpha\}:S_X\to S_Y$ with
the limit $f$ such that 1) $f_\alpha$ is 0-soft for every $\alpha$; 2) every
limit square diagram is 0-soft.
\end{theoremA}

\begin{theorem}\label{metrsoft} Let $X\subset\R^\omega$ be a compact Max-Plus convex subset and $f : X \to Y$ be
an open map onto a compact metrizable space with Max-Plus convex
preimages. Then the map $f$ is soft.
\end{theorem}

\begin{proof} The theorem can be proved using the same arguments as in \cite{Zar2}, where the statement of the theorem was proved  for finite-dimensional $X$.
\end{proof}

A Max-Plus convex compactum $K$ is said to be
I-barycentrically soft (open), if the idempotent barycenter map $b_K$ is soft (open). It is easy to see that the idempotent barycenter map has Max-Plus convex preimages. Hence we obtain the following corollary.

\begin{corollary}\label{metr} Each metrizable I-barycentrically  open compactum is I-barycentrically soft.
\end{corollary}

Now we investigate non-metrizable compacta. It was proved in \cite{Zar} that the functor $I$ preserves open maps, i.e.  the openness of a map $f:X\to Y$ implies  the openness of the map $If:IX\to IY$. We will need the converse statement.

\begin{lemma}\label{refl}  Let $f:X\to Y$ be a continuous map such that  the map $If:IX\to IY$ is open. Then $f$ is open.
\end{lemma}

\begin{proof} Suppose the contrary. Then there exists a point $x\in X$, a neighborhood $U$ of $x$ and a net $(y_\alpha)_{\alpha\in A}$ converging to $f(x)$ such that $f^{-1}(y_\alpha)\cap U=\emptyset$. Then we have that the net $(\delta_{y_\alpha})_{\alpha\in A}$ converges to $\delta_{f(x)}=Ff(\delta_x)$. Take a function $\psi\in C(X)$ such that $\psi(x)=1$ and $\psi(X\setminus U)\subset\{0\}$. We have $\delta_x(\psi)=1$. Since the functor $I$ preserves preimages \cite{Zar}, we have $(If)^{-1}(\delta_{y_\alpha})\subset I(X\setminus U)$ for each $\alpha\in A$. Hence $\nu(\psi)=0$ for each $\nu\in(If)^{-1}(\delta_{y_\alpha})$ and $\alpha\in A$. We obtain a contradiction to openness of the map $If$.
\end{proof}

\begin{corollary}\label{OG} The compactum $IX$ is openly generated  if and only if $X$ is openly generated.
\end{corollary}

\begin{lemma}\label{OD}  Let $f:X\to Y$ be a Max-Plus affine surjective map between Max-Plus convex compacta $X$ and $Y$ such that the diagram

$$ \CD
IX  @>If>>   IY \\
@VV\beta_XV      @VV\beta_YV  \\
X  @>f>>   Y
\endCD
$$
is open. Then $f$ is open.
\end{lemma}

\begin{proof} Suppose the contrary. Then there exists a point $x\in X$, a neighborhood $U$ of $x$ and a net $(y_\alpha)_{\alpha\in A}$ converging to $f(x)=y$ such that $f^{-1}(y_\alpha)\cap U=\emptyset$.
(By $\exp X$ we denote the hyperspace of $X$, i.e., the set of nonempty
closed subsets of $X$ endowed with Vietoris topology).  We can assume that $f^{-1}(y_\alpha)$ converges to $A\in\exp X$. Since $f$
is a closed map we have that $A$ is a   subset in $f^{-1}(y)$. Evidently, $x\notin A$.
Choose a point $x_1\in A$. There exists $s\in(-\infty,0)$ such that $s\odot x_1\oplus x \notin A$. Consider any open  set $V\supset A$ such that $s\odot x_1\oplus x \notin \Cl V$. We can assume that $f^{-1}(y_\alpha)\subset V$ for every $\alpha\in\B$. Consider $\delta_{s\odot x_1\oplus x} \in IX$. Then $\chi(\delta_{s\odot x_1\oplus x})=(s\odot x_1\oplus x;\delta_{y})$.

For each $\alpha\in\B$ choose a point $x_\alpha\in f^{-1}(y_\alpha)$ such that the net $x_\alpha$ converges to $x_1$.  We have that the net $s\odot x_\alpha\oplus x$ converges to $s\odot x_1\oplus x$ in $X$ and the net $s\odot \delta_{y_\alpha}\oplus \delta_y$ converges to $\delta_y$ in  $IY$. Moreover, $(s\odot x_\alpha\oplus x;s\odot \delta_{y_\alpha}\oplus \delta_y)\in X\times_T IT$.

Choose a  function $\varphi\in C(X)$ such that $\varphi(\Cl V)\subset\{-s+1\}$ and $\varphi(s\odot x_1\oplus x)=0$. Consider the neighborhood $O=\{\nu\in IX| \nu(\varphi)<\frac12\}$ of $\delta_{s\odot x_1\oplus x}$. Let $\mu\in If^{-1}(s\odot \delta_{y_\alpha}\oplus \delta_y)$. By Lemma \ref{supp+} we have $\mu=s\odot\eta\oplus \nu$ where $ \eta\in I(f^{-1}(y_\alpha))$ and $\nu\in I(f^{-1}(y))$. Hence $\mu(\varphi)=1>\frac12$ and $\mu\notin O$.
We obtain a contradiction to openness of the characteristic map $\chi$.
\end{proof}

\begin{theorem}\label{softOG} Let $K$ be a Max-Plus convex compactum such that the map $\beta_K$ is 0-soft. Then $K$ is openly generated.
\end{theorem}

\begin{proof} Present $K$ as a limit of an
$\omega$-system $S_K=\{K_\alpha,p_\alpha,\A\}$ where $K_\alpha$ are Max-Plus convex metrizable
compacta and bonding maps $p_\alpha$ are Max-Plus affine for every $\alpha\in\A$. If
the map $b_K:I(K)\to K$ is 0-soft, then, using the spectral theorem of
E.V.~Shchepin \cite{Shchep} and theorem A, we obtain that there exists a
closed cofinal subset $\B\subset\A$ such that for each $\alpha\in\B$ the
diagram
$$ \CD I(K)     @>I(p_\alpha)>>    I(K_\alpha)  \\
@VVb_KV @VVb_{K_\alpha}V
\\ K       @>p_\alpha>>         K_\alpha \endCD $$
is 0-soft and therefore open. It follows from Lemma \ref{OD} that the map $p_\alpha$ is open
for each $\alpha\in\B$. But since $K=\lim\{K_\alpha,p_\alpha,\B\}$, the
compactum $K$ is openly generated. The theorem is proved.
\end{proof}

\begin{lemma}\label{noniz}  Let compactum $X$ be a limit space of  an
$\omega$-system $S_X=\{X_\alpha,p_\alpha,\A\}$ and $x\in X$ has uncountable character. Then  there exists $\alpha\in \A$ such that the point $p_\alpha(x)$ is non-isolated in $X_\alpha$.
\end{lemma}

\begin{proof} Take any $\alpha_1\in\A$. Since $X_{\alpha_1}$ is metrizable, the set  $p_{\alpha_1}^{-1}(p_{\alpha_1}(x))$ contains more then one point. There exists $\alpha_2\in A$ such that $(p^{\alpha_2}_{\alpha_1})^{-1}(p_{\alpha_1}(x))$ contains more then one point. Inductively we can find a sequence $\alpha_1\le\alpha_2\le\dots\le\alpha_i\le\dots$ such that $(p^{\alpha_{i+1}}_{\alpha_i})^{-1}(p_{\alpha_i}(x))$ contains more then one point for each $i\in\N$. Since the system $S_X$ is $\omega$-complete, there exists $\alpha=\sup_{i\in\N}\alpha_i$. Since the system $S_X$ is continuous, $X_\alpha=\lim\{X_{\alpha_i},p^\alpha_{\alpha_i},\N\}$ and the point $p_\alpha(x)$ is non-isolated in $X_\alpha$.
\end{proof}

\begin{theorem}\label{softmetr} Let $X$ be a  compactum such that the map $\beta_{IX}$ is 0-soft. Then $X$ is metrizable.
\end{theorem}

\begin{proof} Theorem \ref{softOG} implies that $IX$ is openly generated. Then $X$ is openly generated by Corollary \ref{OG}.

Suppose that $X$ is non-metrizable.  Present $X$ as a limit of an
$\omega$-system $S_X=\{X_\alpha,p_\alpha,\A\}$ with open surjective maps $p_\alpha$. If
the map $b_{IX}:I^2(X)\to IX$ is 0-soft, then, using the spectral theorem of
E.V.~Shchepin \cite{Shchep} and theorem A, we obtain that there exists a
closed cofinal subset $\B\subset\A$ such that for each $\alpha\in\B$ the
diagram
$$ \CD I^2(X)     @>I^2(p_\alpha)>>    I^2(X_\alpha)  \\
@VVb_{IX}V @VVb_{I(X_\alpha)}V
\\ IX       @>I(p_\alpha)>>         I(X_\alpha) \endCD $$
is 0-soft and therefore open.

Since $X$ is non-metrizable, there exists a point $x\in X$ with uncountable character. By Lemma \ref{noniz}, there exists $\alpha\in\B$ such that the point $y=p_\alpha(x)$ is non-isolated. Since $X_\alpha$ is metrizable, the set $p_\alpha^{-1}(y)$ is not a one-point set. Then there exists $\beta\in\A$ such that $\beta\ge\alpha$ and the set $(p^\beta_\alpha)^{-1}(y)$ is not a one-point set. The characteristic map $\chi:I^2(X_\beta)\to I(X_\beta)\times_{X_\alpha}I^2(X_\alpha)$ of the diagram
$$ \CD I^2(X_\beta)     @>I^2(p^\beta_\alpha)>>    I^2(X_\alpha)  \\
@VVb_{I(X_\beta)}V @VVb_{I(X_\alpha)}V
\\ I(X_\beta)       @>I(p^\beta_\alpha)>>         I(X_\alpha) \endCD $$
is open being a left divisor of the open map $(I(p_\beta)\times\id_{I^2(X_\alpha)}|_{(I(X_\beta)\times_{X_\alpha}I^2(X_\alpha))})\circ \chi'$ where $\chi'$ is the characteristic map of the diagram
$$ \CD I^2(X)     @>I^2(p_\alpha)>>    I^2(X_\alpha)  \\
@VVb_{IX}V @VVb_{I(X_\alpha)}V
\\ IX       @>I(p_\alpha)>>         I(X_\alpha). \endCD $$

Choose two distinct point $x_1$ and $x_2\in(p^\beta_\alpha)^{-1}(y)$. Consider $\delta_{\delta_{x_1}}\oplus\delta_{\delta_{x_2}}\in I^2(X_\beta)$. Then we have $\chi(\delta_{\delta_{x_1}}\oplus\delta_{\delta_{x_2}})=(\delta_{x_1}\oplus\delta_{x_2};\delta_{\delta_{y}})$.

Choose any converging to $y$ sequence $(y_i)$ such that $y_i\neq y$ for each $i\in \N$. Since the map $p^\beta_\alpha$ is open, there exists a sequence $(x_i)$ converging to $x_2$ such that  $p^\beta_\alpha(x_i)=y_i$.    Then the sequence $\delta_{x_1}\oplus\delta_{x_2}\oplus\delta_{x_i}$ converges to $\delta_{x_1}\oplus\delta_{x_2}$ and the sequence $\delta_{\delta_y\oplus\delta_{y_i}}$ converges to $\delta_{\delta_{y}}$. Moreover, $I(p^\beta_\alpha)(\delta_{x_1}\oplus\delta_{x_2}\oplus\delta_{x_i})=\delta_y\oplus\delta_{y_i}
=b_{I(X_\alpha)}(\delta_{\delta_y\oplus\delta_{y_i}})$, hence we have $(\delta_{x_1}\oplus\delta_{x_2}\oplus\delta_{x_i},\delta_{\delta_y\oplus\delta_{y_i}})\in
I(X_\beta)\times_{X_\alpha}I^2(X_\alpha)$ for each $i\in \N$.

 Consider any $\M_i\in\chi^{-1}(\delta_{x_1}\oplus\delta_{x_2}\oplus\delta_{x_i},\delta_{\delta_y\oplus\delta_{y_i}})$. Since $\M_i\in (b_{IX})^{-1}(\delta_{x_1}\oplus\delta_{x_2}\oplus\delta_{x_i})$, we obtain $\M_i\in I^2(\{x_1,x_2,x_i\})$ by Lemma \ref{supp}. Since $\M_i\in (I^2(p^\beta_\alpha))^{-1}(\delta_{\delta_y\oplus\delta_{y_i}})$ and the functor $I$ preserves preimages, we have $\M_i\in IS$ where $S=\{\nu\in I(\{x_1,x_2,x_i\})|\nu=s\odot \delta_{x_1}\oplus t\odot\delta_{x_2}\oplus\delta_{x_i}$ where $t,s\in [-\infty,0]$ with $s\oplus t=0\}$ by Lemma \ref{supp+}.

Choose a function $\varphi\in C(X_\beta)$ such that $\varphi(x_1)=0$ and $\varphi(x_2)>1$. We can assume that $\varphi(x_i)>1$ for each  $i\in \N$. Consider open sets $O_1=\{\nu\in I(X_\beta)|\nu(\varphi)<\frac{1}{3}\}$ and $O_2=\{\nu\in I(X_\beta)|\nu(\varphi)>\frac{2}{3}\}$. Then we have $\delta_{x_1}\in O_1$,  $S\subset O_2$ and $\Cl O_1\cap \Cl O_2$. Choose a function $\psi\in C(I(X_\beta))$ such that $\psi(O_1)\subset\{1\}$ and $\psi(O_2)\subset\{0\}$. Then we have $\delta_{\delta_{x_1}}\oplus\delta_{\delta_{x_2}}(\psi)=1$ and $\M_i(\psi)=0$ for each  $i\in \N$. We obtain a contradiction to openness of $\chi$.  The theorem is proved.
\end{proof}


Let us remark that an analogous theorem for probability measures  was proved in \cite{Radul2}. Fedorchuk proved in \cite{Fed1} that  each product of $\omega_1$ barycentrically open convex metrizable compacta  is barycentrically soft. The following theorem demonstrates that the situation is different in the case of idempotent probability measures.

\begin{theorem}\label{softpr} The map $\beta_{I([0,1]^{\omega_1})}$ is not 0-soft.
\end{theorem}

\begin{proof} Suppose the contrary. Then using the same arguments as in the proof of Theorem \ref{softmetr} we obtain that the diagram

$$ \CD I([0,1]^{\omega}\times[0,1]^{\omega})     @>I(p)>>    I([0,1]^{\omega})  \\
@VVb_{[0,1]^{\omega}\times[0,1]^{\omega}}V @VVb_{[0,1]^{\omega}}V
\\ [0,1]^{\omega}\times[0,1]^{\omega}       @>p>>         [0,1]^{\omega} \endCD $$
is open (by $p:[0,1]^{\omega}\times[0,1]^{\omega}\to[0,1]^{\omega}$ we denote the natural projection to the second coordinate). As before by $\chi$ we denote the characteristic map.

For $t\in [0,1]$ we put $\overline{t}=(t,0,0,\dots)\in [0,1]^{\omega}$. Consider $\delta_{(\overline{0};\overline{0})}\oplus\delta_{(\overline{1};\overline{1})}\in I([0,1]^{\omega}\times[0,1]^{\omega})$. Then we have $\chi(\delta_{(\overline{0};\overline{0})}\oplus\delta_{(\overline{1};\overline{1})})=
((\overline{1};\overline{1});\delta_{\overline{0}}\oplus\delta_{\overline{1}})$.

The sequence $(\overline{1};\overline{1-\frac{1}{i}})$ converges to $(\overline{1};\overline{1})$ and the sequence $\delta_{\overline{0}}\oplus(-\frac{1}{i})\odot\delta_{\overline{1}}$ converges to $\delta_{\overline{0}}\oplus\delta_{\overline{1}}$. Moreover, $((\overline{1};\overline{1-\frac{1}{i}}),\delta_{\overline{0}}\oplus(-\frac{1}{i})\odot\delta_{\overline{1}})\in
([0,1]^{\omega}\times[0,1]^{\omega})\times_{[0,1]^{\omega}}I([0,1]^{\omega})$ for each $i\in \N$.

 Consider any $\pi_i\in\chi^{-1}((\overline{1};\overline{1-\frac{1}{i}}),
 \delta_{\overline{0}}\oplus(-\frac{1}{i})\odot\delta_{\overline{1}})$.  Since $\pi_i\in (Ip)^{-1}(\delta_{\overline{0}}\oplus(-\frac{1}{i})\odot\delta_{\overline{1}})$, we have by Lemma \ref{supp+} that $\pi_i=\nu_i\oplus(-\frac{1}{i})\odot\mu_i$ where $\nu_i\in I([0,1]^{\omega}\times\{\overline{0}\})$ and  $\mu_i\in I([0,1]^{\omega}\times\{\overline{1}\})$ for each $i\in \N$. Since $\pi_i\in (b_{[0,1]^{\omega}\times[0,1]^{\omega}})^{-1}(\overline{1};\overline{1-\frac{1}{i}})$, we have $d_{\nu_i}(\overline{1};\overline{0})=1$ for each $i\in \N$ where $d_{\nu_i}$ is the density of $\nu_i$.

Choose a function $\varphi\in C([0,1]^{\omega}\times[0,1]^{\omega})$ such that $\varphi(\overline{1};\overline{0})=1$ and $\varphi(\overline{0};\overline{0})=\varphi(\overline{1};\overline{1})=0$.  Then we have $\delta_{(\overline{0};\overline{0})}\oplus\delta_{(\overline{1};\overline{1})}(\varphi)=0$ and $\pi_i(\varphi)\ge\nu_i(\varphi)\ge 1$ for each  $i\in \N$. We obtain a contradiction to openness of $\chi$.  The theorem is proved.
\end{proof}

\begin{question} If there exists a non-metrizable I-barycentrically soft compactum?
\end{question}

\end{document}